\documentclass[10pt,reqno]{amsart}

\usepackage{amssymb,latexsym,amsmath}

\def\sup{\mathop{\rm sup}}
\def\Re{\mathop{\rm Re}}
\def\Dom{\mathop{\rm Dom}}

\newtheorem{theorem}{Theorem}[section]
\newtheorem{lemma}[theorem]{Lemma}
\newtheorem{corollary}[theorem]{Corollary}
\newtheorem{remark}[theorem]{Remark}
\begin{document}

\title[]{Stability of solutions to some abstract evolution equations with delay}

\author{N. S. Hoang}

\address{Mathematics Department, University of West Georgia,
Carrollton, GA 30116, USA}
\email{nhoang@westga.edu}

\author{A. G. Ramm}

\address{Mathematics Department, Kansas State University,
Manhattan, KS 66506, USA}
\email{ramm@ksu.edu}


\date{}


\begin{abstract}
The global existence and stability of the solution to the delay differential equation (*)$\dot{u} = A(t)u + G(t,u(t-\tau)) + f(t)$, $t\ge 0$, $u(t) = v(t)$, $-\tau \le t\le 0$, are studied. 
Here $A(t):\mathcal{H}\to \mathcal{H}$ is a closed, densely defined,  linear operator in a Hilbert space $\mathcal{H}$ and 
$G(t,u)$ is a nonlinear operator in $\mathcal{H}$ continuous with respect to $u$ and $t$. 
We assume that the spectrum of $A(t)$ lies in the half-plane 
$\Re \lambda \le \gamma(t)$, where $\gamma(t)$ is not necessarily negative and $\|G(t,u)\| \le \alpha(t)\|u\|^p$, $p>1$, $t\ge 0$. 
Sufficient conditions for the solution to the equation to exist globally, to be bounded and to converge to zero as $t$ tends to $\infty$, under the non-classical assumption that $\gamma(t)$ can take positive values, are proposed and justified. 
\end{abstract}
\maketitle

\pagestyle{plain}

\section{Introduction}

Consider the following delay differential equation
\begin{subequations}
\label{eq1}
\begin{align}
\label{eq1a}
\dot{u} &= A(t)u + G(t,u(t-\tau)) + f(t),\qquad t\ge 0,\qquad \dot{u} := \frac{d u}{dt},\\
\label{eq1b}
u(t) &= v(t),\qquad -\tau\le t\le 0,\qquad \tau=const>0,\qquad v(t)\in C([-\tau,0];\mathcal{H}).
\end{align}
\end{subequations}
Reference to equation \eqref{eq1} means reference to both equations \eqref{eq1a} and \eqref{eq1b}.
Here, $A(t):\mathcal{H}\to \mathcal{H}$ is a closed, densely defined, linear operator in a Hilbert space $\mathcal{H}$ for any fixed $t\ge 0$,
\begin{equation}
\label{eq2}
\Re\langle u, A(t)u\rangle \le \gamma(t)\|u\|^2,\qquad u\in \Dom(A)\subset \mathcal{H},
\end{equation}
$G(t,u)$ is a nonlinear operator in $\mathcal{H}$ for any fixed $t\ge 0$,
\begin{equation}
\label{eq3}
\|G(t,u)\| \le \alpha(t)\| u\|^p,\qquad p>1,\qquad u\in \mathcal{H},
\end{equation}
and $f(t)$ is a function on $\mathbb{R}_+= [0,\infty)$ with values in $\mathcal{H}$,
\begin{equation}
\label{eq4}
\|f(t)\| \le \beta(t),\qquad t\ge 0.
\end{equation}
Here, $\langle \cdot,\cdot\rangle$ and $\|\cdot\|$ denote the inner product and the norm in $\mathcal{H}$, respectively. The functions $\gamma(t)$ and $\alpha(t)$ are continuous on $[0,\infty)$ and real-valued.

Functional differential equations have been studied extensively in the literature (see, e.g., \cite{ref1}--\cite{ref4} and references therein). 
The usual assumption to derive the global existence and the stability of the solution to equation \eqref{eq1} is: $\gamma(t)\le\gamma_0<0$, $\forall t\ge 0$. 
If $A(t)$ is a square matrix, then the condition $\gamma(t)\le\gamma_0<0$ implies that 
all the eigenvalues of $A(t)$ lie in the half-plane $\Re \lambda \le \gamma_0<0$. 
In \cite{ref5}, \cite{ref6}, and \cite{R4}, stability of solution to abstract differential equation \eqref{eq1} when $\tau =0$, i.e., without delay, was studied for the cases $0<\gamma(t)\searrow 0$ and $0>\gamma(t)\nearrow 0$. The main tool for the study of the stability in \cite{ref5}, \cite{ref6}, and \cite{R4} under these non-classical assumptions is some nonlinear inequalities. 
These inequalities were also used in the study of the Dynamical Systems Method (DSM) for solving operator equations in \cite{ref7}. 
In \cite{hoang14} the stability of the solution to equation \eqref{eq1} with $\tau=0$, i.e., without delay, was studied for the case when $\gamma(t)$ can take positive and negative values.
In \cite{hoang14} the nonlinear inequalities were not used, in contrast to \cite{ref5}, \cite{ref6}, and \cite{R4}. 
Using a nonlinear inequality with delay, the global existence and stability of equation \eqref{eq1} were studied in \cite{R5} for the case when $f(t) = 0$ and $G(t,u)$ is of the form $B(t)F(t,u)$ under the non-classical assumption that \eqref{eq2} holds but the inequality $\gamma(t) \le \gamma_0<0$ does not hold for any $\gamma_0<0$. 
In this paper, we are interested in having the stability results for the solution to equation \eqref{eq1} without using nonlinear inequalities similar to those used in \cite{R5}.

A common approach to obtain the global existence of the solution to equation 
\eqref{eq1} is to estimate $\|u(t)\|$ for $t\ge 0$ and use the local existence of the solution to extend the existence of the solution to $[0,\infty)$. An estimate of $\|u(t)\|$ for $t\ge 0$ can be derived from a nonlinear inequality 
(see inequality \eqref{eq6} below) which is obtained from equation \eqref{eq1} as follows. 
Take the inner product of both sides of equation \eqref{eq1a} with $u$ to get
\begin{equation}
\label{eq5.0}
\begin{split}
\langle u(t), \dot{u}(t)\rangle  &= \langle u(t), A(t)u(t)  +  G(t,u(t - \tau)) + f(t)\rangle\\
& = \langle u(t), A(t)u(t) \rangle  + \langle u(t), G(t,u(t - \tau)) \rangle + \langle u(t), f(t)\rangle,
\qquad t\ge 0. 
\end{split}
\end{equation}
Denote $g(t) := \|u(t)\|$, take the real part of equation \eqref{eq5.0} and use the triangle inequality, the Cauchy--Schwarz inequality, and inequalities \eqref{eq2}--\eqref{eq4} to get
\begin{equation}
\label{eq5}
\dot{g}(t)g(t) \le \gamma(t)g^2(t) + \alpha(t)g(t)g^p(t-\tau) + g(t)\beta(t),\qquad t\ge 0.
\end{equation}
The derivative $\dot{g}(t)$ in \eqref{eq5} is understood as the right derivative at $t$ if $g(t)=0$. 
From inequality \eqref{eq5} and equation \eqref{eq1b}  one gets 
\begin{subequations}
\label{eq6}
\begin{align}
\label{eq6a}
\dot{g}(t) &\le \gamma(t)g(t) + \alpha(t)g^{p}(t-\tau) + \beta(t),\qquad t\ge 0,\\
g(t) &= w(t):= \|v(t)\|,\qquad -\tau\le t\le 0.
\end{align}
\end{subequations}
If $g(t)>0$, $\forall t\ge 0$, then it is clear that \eqref{eq6a} follows from \eqref{eq5}. If $g(t) = 0$ for some $t>0$, then
these $t$ form a set of isolated points by the uniqueness theorem
and therefore equations
 \eqref{eq5} and \eqref{eq6a} are equivalent. This equivalence is proved differently in Lemma \ref{lemma4} below. 
By studying the global existence and boundedness of a solution $g(t)$ to inequality \eqref{eq6} and using the local existence of the solution $u(t)$ to equation \eqref{eq1}, one can obtain the global existence and boundedness of $u(t)$.

In \cite{R5} the stability of the solution to the following differential equation with delay was studied:
\begin{subequations}
\label{eqx1}
\begin{align}
\dot{u} &= A(t)u + B(t)F(t,u(t-\tau)),\qquad t\ge 0,\qquad \dot{u} := \frac{d u}{dt},\\
u(t)&= v(t),\qquad -\tau\le t\le 0,\qquad \tau=const>0.
\end{align}
\end{subequations}
It was assumed in \cite{R5} that $A(t)$ and $B(t)$ are linear operators in a Hilbert space $\mathcal{H}$ and $F(t,u)$ is a nonlinear operator in $\mathcal{H}$ for any fixed $t\ge 0$. It was also assumed in \cite{R5} that relation \eqref{eq2} holds, $\|B(t)\| \le b(t)$, and $\|F(t,u)\|\le \alpha(t,\|u\|)$.  
It is clear that equation \eqref{eqx1} is a special case of equation \eqref{eq1} when the function $f(t)$ vanishes and $G(t,u) = B(t)F(t,u)$. 
Using the assumptions on $A(t)$, $B(t)$, and $F(t,u)$ and an assumption on the existence of a local solution to problem \eqref{eqx1}, the following result was proved in \cite{R5}: 

{\em If there exists a function $\mu(t)>0$, defined for all $t\ge -\tau$, such that
\begin{subequations}
\label{eqx2}
\begin{align}
b(t)\alpha\bigg(t,\frac{1}{\mu(t-\tau)}\bigg)\mu(t) &\le -\gamma(t) - \frac{\dot{\mu}(t)}{\mu(t)},\qquad t\ge 0,\\
\|u(t)\| &\le \frac{1}{\mu(t)},\qquad t\in [-\tau,0],
\end{align}
\end{subequations}
then the solution to \eqref{eqx1} exists for all $t\ge 0$ and
$$
\|u(t)\| \le \frac{1}{\mu(t)},\qquad t\ge 0.
$$}
Using this result the global existence and the boundedness of the solution to equation 
\eqref{eqx1} were obtained for some classes of functions $b(t), \gamma(t)$, and $\alpha(t,y)$ (see \cite{R5}).

Although the mentioned result in \cite{R5} is quite general, it requires to find a function $\mu(t)$ which solves inequality \eqref{eqx2}. In general, it is not easy to find such function $\mu(t)$ if the functions $b(t), \gamma(t)$, and $\alpha(t,y)$ are not simple. 

In this paper we are interested in having stability results for equation \eqref{eq1} under non-classical assumptions on the operator $A(t)$, namely that the function $\gamma(t)$ in \eqref{eq2} may  change signs.  
In particular, we want find  
sufficient conditions on the functions $\gamma(t)$, $\alpha(t)$, and $\beta(t)$ (see \eqref{eq2}--\eqref{eq4}) which yield the global existence and the boundedness of the solution to equation \eqref{eq1}.  

The main results of this paper are Theorem \ref{thm1}, Theorem \ref{thm2}, and Corollary \ref{corollary1}. 
In Theorems \ref{thm1} and \ref{thm2} sufficient conditions on $\gamma(t)$, $\alpha(t)$, and $\beta(t)$ for the solution of equation \eqref{eq1} to exist globally, to be bounded, and to decay to zero as $t\to \infty$, are given. A direct consequence of Theorem \ref{thm1} for the case $f=0$ is formulated in Corollary \ref{corollary1}. The novelties of our results compared to those in \cite{R5} include: 

{\em Our results do not require one to find the function $\mu(t)>0$ which solves the nonlinear inequality \eqref{eqx2}.}

 Thus, these results are applicable when the functions $\gamma(t)$, $\alpha(t)$, and $\beta(t)$ are quite general. Moreover, our results cover the case when $f\not= 0$ which was not considered in \cite{R5}. 

Throughout this paper, we suppose that the following assumption holds:

{\bf Assumption A:}

\noindent{\it The equation
\begin{subequations}
\label{eq7}
\begin{align}
\dot{u} &= A(t)u + G(t,u(t-\tau)) + f(t),\qquad t\ge a,\qquad 
\dot{u} := \frac{d u}{dt},\\
u(t) &= v_{a}(t),\qquad a-\tau\le t\le a,\qquad 
v_a \in C([a-\tau,a];\mathcal{H}),
\end{align}
\end{subequations}
has a unique local solution for all $a\ge 0$ and $v_a \in C([a-\tau,a];\mathcal{H})$.}

It is known that Assumption A holds if $A(t)$ is a generator of a $C_0$ semigroup and the functions $\alpha(t)$ and $\beta(t)$ are 
continuous and bounded on $[0,\infty)$ (see, e.g., \cite{ref1}).

\section{Main results}

Let us first show that inequalities \eqref{eq5} and \eqref{eq6}  are equivalent. 
\begin{lemma}
\label{lemma4}
Let $g(t)\ge 0$ be a solution to
\begin{subequations}
\begin{align}
\label{eq71}
g(t)\dot{g}(t) &\le \gamma(t)g^2(t) + \alpha(t)g(t)g^p(t-\tau) + g(t)\beta(t),\qquad t\ge 0,\\
g(t) &= w(t)\ge 0,\qquad -\tau \le t\le 0.
\end{align}
\end{subequations}
Then $g(t)$ solves the inequality
\begin{subequations}
\begin{align}
\label{eq72}
\dot{g}(t) &\le \gamma(t)g(t) + \alpha(t)g^p(t-\tau) + \beta(t),\qquad t\ge 0,\\
g(t) &= w(t),\qquad -\tau \le t\le 0.
\end{align}
\end{subequations}
\end{lemma}

\begin{proof}
Equation \eqref{eq71} can be written as
\begin{equation}
\label{eqz1}
\frac{d}{dt}\, g^2(t) \le 2\gamma(t)g^2(t) + 2\alpha(t)g(t)g^p(t-\tau) + 2g(t)\beta(t),\qquad t\ge 0.
\end{equation}
Let 
\begin{equation}
\label{eqz2}
g_\epsilon(t) := \bigg[g^2(t) + \epsilon e^{2\int_0^t \gamma(\xi)\, d\xi} \bigg]^\frac{1}{2},\qquad t\ge -\tau,\qquad \epsilon=const>0. 
\end{equation}
From \eqref{eqz1} and \eqref{eqz2} one gets
\begin{equation}
\label{eqz3}
\begin{split}
\frac{d}{dt}\, g^2_\epsilon(t) &= \frac{d}{dt}\, g^2(t)  + 2\epsilon \gamma(t) e^{2\int_0^t \gamma(\xi)\, d\xi} \\
&\le 2\gamma(t)g^2(t) + 2\alpha(t)g(t)g^p(t-\tau) + 2g(t)\beta(t) + 2\epsilon \gamma(t) e^{2\int_0^t \gamma(\xi)\, d\xi} \\
&= 2\gamma(t)g^2_\epsilon(t) + 2\alpha(t)g(t)g^p(t-\tau) + 2g(t)\beta(t) \\
& \le 2\gamma(t)g^2_\epsilon(t) + 2\alpha(t)g_\epsilon(t)g^p_\epsilon(t-\tau) + 2g_\epsilon(t)\beta(t),\qquad t\ge 0. 
\end{split}
\end{equation}
Here the inequality $g_{\epsilon}(t)>g(t)$, $\forall t\ge -\tau$, $\epsilon>0$, was used. Inequality \eqref{eqz3} implies
\begin{equation}
\label{eq74}
g_\epsilon(t)\dot{g}_\epsilon(t) \le \gamma(t)g^2_\epsilon(t) + \alpha(t)g_\epsilon(t) g^p_\epsilon(t-\tau) + g_\epsilon(t)\beta(t),\qquad t\ge 0. 
\end{equation}
Since $g_\epsilon(t)\ge \sqrt{\epsilon}e^{\int_0^t \gamma(\xi)\, d\xi}>0$, $\forall t\ge 0$, $\epsilon>0$, by \eqref{eqz2}, it follows from inequality \eqref{eq74} that
\begin{equation}
\label{eq75}
\dot{g}_\epsilon(t) \le \gamma(t)g_\epsilon(t) + \alpha(t)g^p_\epsilon(t-\tau) + \beta(t),\qquad t\ge 0,\quad \epsilon > 0. 
\end{equation}
Since $g(t)=\lim_{\epsilon\to 0}g_\epsilon(t)$ by \eqref{eqz2},
and $\lim_{\epsilon\to 0}\dot{g}_\epsilon(t)=\dot{g}(t)$,
 inequality \eqref{eq72} follows from \eqref{eq75} by letting $\epsilon \to 0$. Lemma \ref{lemma4} is proved. 
\end{proof}

Consider the following delay differential equation
\begin{subequations}
\label{eq8}
\begin{align}
\label{eq8a}
\dot{h}(t) &= \gamma(t)h(t) + \alpha(t)h^{p}(t-\tau) + \beta(t),\qquad t\ge 0,\\
h(t) &= w(t),\qquad -\tau\le t\le 0.
\end{align}
\end{subequations}
We have the following comparison lemma. 

\begin{lemma}
\label{compare}
Let $g(t)$ solve inequality \eqref{eq6} and $h(t)$ solve equation \eqref{eq8}. Then
\begin{equation}
\label{eq9}
g(t) \le h(t),\qquad \forall t\in [0,\tilde{T}),
\end{equation}
where $[-\tau,\tilde{T})$ is the maximal interval of the existence of $h(t)$. 
\end{lemma}

\begin{proof}
Let $h_n(t)$ be the solution to the equation
\begin{subequations}
\label{eqz4}
\begin{align}
\label{eqz4a}
\dot{h}_n(t) &= \gamma(t)h_n(t) + \alpha(t)h_n^{p}(t-\tau) + \beta(t) + \frac{1}{n},\qquad t\ge 0,\qquad n=1,2,....,\\
\label{eqz4b}
h_n(t) &= w(t),\qquad -\tau\le t\le 0.
\end{align}
\end{subequations}
Since $g(t)=w(t)=h_n(t)$, $\forall t\in [-\tau,0]$, it follows from \eqref{eq6a} and \eqref{eqz4a} that
\begin{equation*}
\begin{split}
g'(0) & \le  \gamma(0)g(0) + \alpha(0)g^{p}(-\tau) + \beta(0)\\
&= \gamma(0)h_n(0) + \alpha(0)h_n^{p}(-\tau) + \beta(0) = h_n'(0) - \frac{1}{n} < h_n'(0). 
\end{split}
\end{equation*}
This and the equality $g(0)=h_n(0)$ imply the existence of $\delta>0$ such that $g(t)<h_n(t)$, $\forall t\in (0,\delta)$. Let $T>0$ be the largest value such that
$$
g(t) < h_n(t),\qquad \forall t\in (0,T).
$$
{\it We claim that $T=T_n$ where $[-\tau,T_n)$ is the maximal interval of the existence of $h_n(t)$.}

To prove this claim, assume the contrary. Then $T<T_n$ and $h_n(T) = g(T)$, by the continuity of $g(t)$ and $h_n(t)$ and the definition of $T$. Since $g(t) < h_n(t)$, $\forall t\in (0, T)$, and $g(T) = h_n(T)$, one concludes that 
\begin{equation}
\label{eqdh1}
\dot{g}(T)\ge \dot{h}_n(T). 
\end{equation}
On the other hand, from \eqref{eq6a}, \eqref{eqz4a}, the equality $g(T)=h_n(T)$, and the inequality $g(T-\tau) \le h(T-\tau)$, we have
\begin{equation*}
\begin{split}
\dot{g}(T) &\le \gamma(T)g(T) + \alpha(T)g^{p}(T-\tau) + \beta(T)\\
&= \gamma(T) h_n(T) + \alpha(T)h_n^{p}(T-\tau) + \beta(T) = \dot{h}_n(T)-\frac{1}{n}<\dot{h}(T). 
\end{split}
\end{equation*}
This contradicts to inequality \eqref{eqdh1} and implies that $T=T_n$, i.e.,
$$
g(t) < h_n(t),\qquad \forall t\in (0,T_n).
$$
Since $h(t) = \lim_{n\to\infty}h_n(t)$ and $\lim_{n\to\infty}T_n = \tilde{T}$, one gets $g(t)\le h(t)$, $\forall t\in [0,\tilde{T})$. Lemma \ref{compare} is proved. 
\end{proof}

Lemma \ref{compare}  says that for any solution $g(t)\ge 0$ to inequality \eqref{eq6} one has 
$0\le g(t)\le h(t)$. Thus, the global existence and boundedness of $h(t)$ imply the global existence and boundedness of $g(t)$. 
Therefore, to study the global existence and boundedness of $g(t)$, we will study the global existence and boundedness of the solution $h(t)$ to equation \eqref{eq8}.

Let us consider equation \eqref{eq8}. 
Since the functions $h(t)$, $\alpha(t)$, and $f(t)$ are nonnegative on $\mathbb{R}_+$, it follows from \eqref{eq8a} that
\begin{equation}
\label{eq10}
\dot{h}(t) \ge \gamma(t)h(t),\qquad t\ge 0.
\end{equation}
This inequality is equivalent to
$$
\frac{d}{dt}\bigg[h(t)e^{-\int_0^t \gamma(\xi)\, d\xi}\bigg] \ge 0,\qquad t\ge 0.
$$ 
Integrate this inequality from $t-\tau$ to $t$ to get
\begin{equation}
\label{eq11}
h(t)e^{-\int_0^{t} \gamma(\xi)\, d\xi} \ge h(t-\tau)e^{-\int_0^{t-\tau} \gamma(\xi)\, d\xi},\qquad t\ge \tau. 
\end{equation}
Thus,
\begin{equation}
\label{eq11'}
h(t)\sigma(t) \ge h(t-\tau),\qquad \sigma(t):= e^{-\int_{t-\tau}^{t} \gamma(\xi)\, d\xi},\qquad t\ge \tau.
\end{equation}
This and inequality \eqref{eq8a} imply
\begin{equation}
\label{eq12}
\dot{h}(t) \le \gamma(t)h(t) + \alpha(t)\sigma^p(t)h^{p}(t) + \beta(t),\qquad t\ge \tau.
\end{equation}

It follows from equation \eqref{eq8a} that
\begin{equation}
\label{eq13}
\frac{d}{dt}\bigg[ h(t)\nu(t)\bigg ] = \nu(t) \alpha(t) h^p (t-\tau) + \nu(t)\beta(t),\qquad t\ge 0,\qquad \nu(t) := e^{-\int_0^t \gamma(\xi)\, d\xi}.
\end{equation}
Integrate this equation from $0$ to $\tau$ to get
$$
h(\tau)\nu(\tau) - h(0)\nu(0) = \int_0^\tau \big[\alpha(\xi)\nu(\xi)h^p(\xi-\tau) + \beta(\xi)\nu(\xi)\big]\,d\xi.
$$
This and the relation $\nu(0) = 1$ imply
\begin{equation}
\label{eq14}
h(\tau) = \frac{h(0)}{\nu(\tau)} + \frac{\int_0^\tau \big[\alpha(\xi)\nu(\xi)\|v(\xi-\tau)\|^p + \beta(\xi)\nu(\xi)\big]\,d\xi}{\nu(\tau)}. 
\end{equation}

From \eqref{eq12} and \eqref{eq14} one concludes that the solution $h(t)$ to equation \eqref{eq8} satisfies the following inequality
\begin{subequations}
\label{eqz27}
\begin{align}
\label{eqz27a}
\dot{h}(t) &\le \gamma(t)h(t) + \alpha(t)\sigma^p(t)h^{p}(t) + \beta(t),\qquad t\ge \tau,\qquad h(t)\ge 0,\\
\label{eqz27b}
h(\tau) &= h_\tau
\end{align}
\end{subequations}
where
\begin{equation}
\label{eqz28}
h_\tau:=\frac{\|v(0)\|}{\nu(\tau)} + \frac{\int_0^\tau \big[\alpha(\xi)\nu(\xi)\|v(\xi-\tau)\|^p + \beta(\xi)\nu(\xi)\big]\,d\xi}{\nu(\tau)}. 
\end{equation}
By our assumptions \eqref{eq3} and \eqref{eq4} it follows that 
$\alpha(t)$ and $\beta(t)$ are positive. Also $\nu(t)>0$. Therefore $h_\tau>0$. The function $\sigma(t)$, defined in \eqref{eq11'}, is also positive.

The following lemma gives a sufficient condition for the solution to equation \eqref{eq8} to exist globally. 
\begin{lemma}
\label{lemma1}
Let $h(t)$ be a solution to equation \eqref{eq8}. 
Assume that
\begin{gather}
\label{eq29zx}
\frac{1}{\big[h_\tau\nu(\tau) + \omega\big]^{p-1}} > (p-1) \int_\tau^\infty \frac{\alpha(\xi) \sigma^p(\xi)}{\nu^{p-1}(\xi)}\, d\xi,\qquad \nu(t) := e^{-\int_0^t \gamma(\xi)\,d\xi},\\
\label{eq16}
\frac{\beta(t)\nu^p(t)}{\alpha(t)\sigma^p(t)} \le \omega^p,\qquad t\ge \tau,\qquad \omega = const>0,
\end{gather}
where $h_\tau$ is defined in \eqref{eqz28}. 
Then $h(t)$ exists globally and satisfies the estimate
\begin{equation}
\label{eq17}
h(t) \le \frac{\bigg(\frac{1}{\big[h_\tau\nu(\tau) + \omega\big]^{1-p} - (p-1)\int_\tau^t \frac{\alpha(\xi)\sigma^p(\xi)}{\nu^{p-1}(\xi)}\, d\xi}\bigg)^{\frac{1}{p-1}} - \omega}{\nu(t)},\qquad t\ge \tau. 
\end{equation}
\end{lemma}

\begin{remark}
\label{remark2.4}
It follows from inequality \eqref{eq29zx} that the right-hand side of \eqref{eq17} is well-defined for all $t\ge \tau$. 
Inequality \eqref{eq29zx} is equivalent to 
$$
\omega < \frac{1}{\bigg[(p-1)\int_\tau^\infty \frac{\alpha(\xi)\sigma^p(\xi)}{\nu^{p-1}(\xi)}\, d\xi \bigg]^{\frac{1}{p-1}}} - h_\tau\nu(\tau). 
$$
Moreover, inequality \eqref{eq16} is equivalent to  
$$
\bigg(\frac{\beta(t)}{\alpha(t)}\bigg)^\frac{1}{p}e^{-\int_0^{t-\tau}\gamma(\xi)\, d\xi} \le \omega,\qquad t\ge \tau.
$$
Thus, for the existence of $\omega$ satisfying both \eqref{eq29zx} and \eqref{eq16}, it suffices to assume that 
\begin{equation}
\label{eqextra1}
\sup_{t\ge\tau} \bigg(\frac{\beta(t)}{\alpha(t)}\bigg)^\frac{1}{p}e^{-\int_0^{t-\tau}\gamma(\xi)\, d\xi} < \frac{1}{\bigg[(p-1)\int_\tau^\infty \frac{\alpha(\xi)\sigma^p(\xi)}{\nu^{p-1}(\xi)}\, d\xi \bigg]^{\frac{1}{p-1}}} - h_\tau\nu(\tau). 
\end{equation}
Given the function $\gamma(t)$ and the numbers $p$, $\tau$, $h_\tau$ and $\nu(\tau)$, inequality \eqref{eqextra1} holds true if $\alpha(t)$ and $\beta(t)/\alpha(t)$ are sufficiently small. In this case, one can choose
\begin{equation}
\label{eqextra2}
\omega: = \sup_{t\ge\tau} \bigg(\frac{\beta(t)}{\alpha(t)}\bigg)^\frac{1}{p}e^{-\int_0^{t-\tau}\gamma(\xi)\, d\xi}.
\end{equation}
In the case when $f(t)$ is absent from equation \eqref{eq1}, i.e., $\beta(t)=0$, then inequality \eqref{eqextra1} becomes
$$
h_\tau\nu(\tau) < \frac{1}{\bigg[(p-1)\int_\tau^\infty \frac{\alpha(\xi)\sigma^p(\xi)}{\nu^{p-1}(\xi)}\, d\xi \bigg]^{\frac{1}{p-1}}}
$$
and one can choose $\omega = 0$.
\end{remark}

\begin{proof}[Proof of Lemma \ref{lemma1}]
Since $h(t)$ is the solution to \eqref{eq8}, it satisfies inequality \eqref{eqz27a}. 
From \eqref{eqz27a} and \eqref{eq16}, one gets
\begin{equation}
\label{eq18}
\begin{split}
\dot{h} &\le \gamma(t)h(t) + \frac{\alpha(t)\sigma^p(t)}{\nu^p(t)}\bigg[ \nu^p(t)h^p(t) + \frac{\beta(t)\nu^p(t)}{\alpha(t)\sigma^p(t)} \bigg]\\
&\le \gamma(t)h(t) + \frac{\alpha(t)\sigma^p(t)}{\nu^p(t)}\bigg[ \nu^p(t)h^p(t) + \omega^p \bigg]\\
&\le \gamma(t)h(t) + \frac{\alpha(t)\sigma^p(t)}{\nu^p(t)}\bigg[ \nu(t)h(t) + \omega\bigg]^p,\qquad t\ge \tau,\qquad p>1.
\end{split}
\end{equation}
Here, the inequality $a^p+b^p \le (a+b)^p$, $a\ge 0$, $b\ge 0$, $p>1$, was used. 
Inequality \eqref{eq18} can be written as
\begin{equation}
\label{eq19}
\frac{d}{dt}\bigg(h(t)\nu(t) + \omega\bigg) \le \frac{\alpha(t)\sigma^p(t)}{\nu^{p-1}(t)} \bigg[ \nu(t)h(t) + \omega\bigg]^p,\qquad t\ge \tau. 
\end{equation}
Thus,
\begin{equation}
\label{eq20}
\frac{1}{1-p}\frac{d}{dt}\bigg[h(t)\nu(t) + \omega\bigg]^{1-p} \le \frac{\alpha(t)\sigma^p(t)}{\nu^{p-1}(t)},\qquad t\ge \tau. 
\end{equation}
Integrate inequality \eqref{eq20} from $\tau$ to $t$ to get
\begin{equation}
\label{eq21}
\frac{\big[h(t)\nu(t)+\omega\big]^{1-p} - \big[h(\tau)\nu(\tau)+\omega\big]^{1-p}}{1-p} \le \int_{\tau}^t \frac{\alpha(\xi)\sigma^p(\xi)}{\nu^{p-1}(\xi)}\, d\xi,\qquad t\ge \tau. 
\end{equation}
This implies
\begin{equation}
\label{eq22}
\big[h(t)\nu(t)+\omega\big]^{1-p} \ge \big[h(\tau)\nu(\tau)+\omega\big]^{1-p} - (p-1)\int_{\tau}^t \frac{\alpha(\xi)\sigma^p(\xi)}{\nu^{p-1}(\xi)}\, d\xi,\qquad t\ge \tau. 
\end{equation}
It follows from \eqref{eq29zx} that the right-hand side of inequality \eqref{eq22} is positive. Thus, from \eqref{eq22} one gets
\begin{equation}
\label{eq23}
\big[h(t)\nu(t)+\omega\big]^{p-1} \le \frac{1}{\big[h_\tau\nu(\tau)+\omega\big]^{1-p} - (p-1)\int_{\tau}^t \frac{\alpha(\xi)\sigma^p(\xi)}{\nu^{p-1}(\xi)}\, d\xi},\qquad t\ge \tau. 
\end{equation}
Inequality \eqref{eq17} follows from \eqref{eq23}. Lemma \ref{lemma1} is proved.
\end{proof}

\begin{theorem}
\label{thm1}
Let Assumption {\bf A} hold. Assume that
\begin{gather}
\label{eq24}
\omega:=\sup_{t\ge\tau} \bigg(\frac{\beta(t)}{\alpha(t)}\bigg)^\frac{1}{p}e^{-\int_0^{t-\tau}\gamma(\xi)\, d\xi} < \frac{1}{\bigg[(p-1)\int_\tau^\infty \frac{\alpha(\xi)\sigma^p(\xi)}{\nu^{p-1}(\xi)}\, d\xi \bigg]^{\frac{1}{p-1}}} - h_\tau\nu(\tau) 
\end{gather}
where
\begin{equation}
\label{eq25}
h_\tau := \frac{\|v(0)\|}{\nu(\tau)} + \frac{\int_0^\tau \big[\alpha(\xi)\nu(\xi)\|v(\xi-\tau)\|^p + \beta(\xi)\nu(\xi)\big]\,d\xi}{\nu(\tau)},
\end{equation}
$\sigma(t) := e^{-\int_\tau^t \gamma(\xi)\,d\xi}$ and $\nu(t) := e^{-\int_0^t \gamma(\xi)\,d\xi}$.
Then the solution to problem \eqref{eq1} exists globally and
\begin{equation}
\label{eq27}
\|u(t)\| \le \frac{1}{\nu(t)}\bigg[ \bigg(\frac{1}{\big[h_\tau\nu(\tau)+\omega\big]^{1-p} - (p-1)\int_\tau^t \frac{\alpha(\xi)\sigma^p(\xi)}{\nu^{p-1}(\xi)}\, d\xi}\bigg)^{\frac{1}{p-1}} - \omega\bigg],\qquad t\ge \tau. 
\end{equation}
In addition, if 
\begin{equation}
\label{eq28}
M:=\sup_{t\ge 0} \int_0^t \gamma(\xi)\, d\xi < \infty,
\end{equation}
then the solution $u(t)$ is bounded. 

If
\begin{equation}
\label{eq29}
\lim_{t\to\infty} \int_0^t \gamma(\xi)\, d\xi = -\infty,
\end{equation}
then
\begin{equation}
\label{eq30}
\lim_{t\to\infty} u(t) = 0.
\end{equation}
\end{theorem}

\begin{proof}
Let $u(t)$ be the solution to \eqref{eq1} and $g(t):=\|u(t)\|$. Then from Lemma \ref{compare} one has 
$$
g(t) \le h(t),\qquad t\ge 0.
$$
It follows from \eqref{eq24} and Remark \ref{remark2.4} that inequalities \eqref{eq29zx} and \eqref{eq16} hold.  
Using Lemma \ref{lemma1} one gets
\begin{equation}
\label{eq31}
h(t) \le \frac{1}{\nu(t)} \bigg[\bigg(\frac{1}{\big[h_\tau\nu(\tau) + \omega\big]^{1-p} - (p-1)\int_\tau^t \frac{\alpha(\xi)\sigma^p(\xi)}{\nu^{p-1}(\xi)}\, d\xi}\bigg)^{\frac{1}{p-1}} - \omega\bigg],\qquad t\ge \tau. 
\end{equation}
Therefore,
\begin{equation}
\label{eq32}
\|u(t)\| = g(t) \le h(t) \le \frac{\bigg(\frac{1}{\big[h_\tau\nu(\tau) + \omega\big]^{1-p} - (p-1)\int_\tau^t \frac{\alpha(\xi)\sigma^p(\xi)}{\nu^{p-1}(\xi)}\, d\xi}\bigg)^{\frac{1}{p-1}} - \omega}{\nu(t)},\qquad t\ge \tau. 
\end{equation}
Thus, inequality \eqref{eq27} holds. 

Inequality \eqref{eq32} implies
\begin{equation}
\label{eq33}
\|u(t)\| \le \frac{C}{\nu(t)} = Ce^{\int_0^t \gamma(\xi)\, d\xi},\qquad t\ge \tau,
\end{equation}
where
\begin{equation}
\label{eq34}
C: = \bigg(\frac{1}{\big[h_\tau\nu(\tau) + \omega\big]^{1-p} - (p-1)\int_\tau^\infty \frac{\alpha(\xi)\sigma^p(\xi)}{\nu^{p-1}(\xi)}\, d\xi}\bigg)^{\frac{1}{p-1}} - \omega. 
\end{equation}
Inequality \eqref{eq33} and Assumption {\bf A} imply that the solution $u(t)$ exists globally. 

If inequality \eqref{eq28} holds, then
\begin{equation}
\label{eq35}
e^{\int_0^t \gamma(\xi)\, d\xi} \le e^{M},\qquad \forall t\ge 0.
\end{equation}
This and inequality \eqref{eq33} imply
\begin{equation}
\label{eq36}
\|u(t)\| \le Ce^M,\qquad t\ge \tau.
\end{equation}
Thus, the solution $u(t)$ is bounded on $\mathbb{R}_+$. 

If relation \eqref{eq29} holds, then one gets
\begin{equation}
\label{eq37}
\lim_{t\to\infty}e^{\int_0^t \gamma(\xi)\, d\xi} = 0. 
\end{equation}
Relation \eqref{eq30} follows from inequality \eqref{eq33} and formula \eqref{eq37}. Theorem \ref{thm1} is proved. 
\end{proof}

It follows from \eqref{eq24} in Theorem \ref{thm1} that if the right-hand side of  \eqref{eq24} is positive, then one can obtain the global existence and boundedness of $u(t)$ by choosing $f(t)$ so that $\beta(t):=\|f(t)\|$ is sufficiently small. This raises the question: Given $\beta(t):=\|f(t)\|$, is this possible to `control' $\alpha(t)$ so that the global existence and boundedness of $u(t)$ are still guaranteed? We will address this question in the following results. 

\begin{lemma}
\label{lemma2}
Let $h(t)$ be the solution to equation \eqref{eq8} and 
\begin{equation}
\label{eq38}
\zeta(t) : = \frac{h(\tau)\nu(\tau)}{\nu(t)} + \frac{\int_\tau^t \beta(\xi)\nu(\xi)\, d\xi}{\nu(t)},\qquad t\ge \tau. 
\end{equation}
Assume that 
\begin{equation}
\label{eq39}
 \alpha(t)\sigma^p(t) \le \frac{(q-1)\beta(t)}{[q\zeta(t)]^p},\qquad t\ge \tau,\quad q>1.
\end{equation}
Then
\begin{equation}
\label{eq40}
h(t) \le q \zeta(t),\qquad t\ge \tau. 
\end{equation}
\end{lemma}

\begin{proof}
From equation \eqref{eq38} with $t=\tau$ and the assumption $q>1$, one gets
\begin{equation}
h(\tau) = \zeta(\tau)<q\zeta(\tau).
\end{equation}
It follows from the continuity of $\zeta(t)$ and $h(t)$ that there exists $\theta>0$ such that
\begin{equation}
h(t) \le q\zeta(t),\qquad \tau \le t\le \tau + \theta. 
\end{equation} 
Let $T>0$ be the largest real value such that 
\begin{equation}
\label{eqx57}
h(t) \le q\zeta(t),\qquad \tau \le t\le T. 
\end{equation}
We claim that $T= \infty$. Assume the contrary. Then $T$ is finite and by the definition of $T$ we have
\begin{equation}
\label{eq44}
h(T) = q\zeta(T). 
\end{equation}
Since $h(t)$ is the solution to \eqref{eq8}, it satisfies inequality \eqref{eqz27a}.
It follows from inequalities \eqref{eqz27a}, \eqref{eqx57}, and \eqref{eq39} that
\begin{equation}
\label{eq45}
\begin{split}
\dot{h}(t) &\le \gamma(t)h(t) + \alpha(t)\sigma^p(t)[q\zeta(t)]^p + \beta(t)\\
&\le \gamma(t)h(t) + (q-1)\beta(t) + \beta(t) = \gamma(t)h(t) + q\beta(t),\qquad \tau \le t\le T.
\end{split}
\end{equation}
From \eqref{eq45} one gets
\begin{equation}
\frac{d}{dt}\big[\nu(t)h(t)\big] \le q \nu(t)\beta(t),\qquad \tau\le t\le T.
\end{equation}
Integrate this inequality from $\tau$ to $t$ to get
\begin{equation}
\nu(t)h(t) - \nu(\tau)h(\tau) \le q\int_\tau^t \beta(\xi)\nu(\xi)\, d\xi,\qquad \tau \le t\le T.
\end{equation}
This and \eqref{eq38} imply
\begin{equation}
\begin{split}
h(t) &\le \frac{\nu(\tau)h(\tau)}{\nu(t)} + \frac{q\int_\tau^t \beta(\xi)\nu(\xi)\, d\xi}{\nu(t)}\\
 &<  q\bigg(\frac{\nu(\tau)h(\tau)}{\nu(t)} + \frac{\int_\tau^t \beta(\xi)\nu(\xi)\, d\xi}{\nu(t)} \bigg)= q\zeta(t),\qquad \tau\le t\le T.
\end{split}
\end{equation}
Therefore,
\begin{equation}
h(T) < q\zeta(T).
\end{equation}
This contradicts to equation \eqref{eq44}. The contradiction implies that $T=\infty$, i.e., inequality \eqref{eq40} holds. Lemma \ref{lemma2} is proved. 
\end{proof}

\begin{theorem}
\label{thm2}
Let Assumption A hold. Assume that
\begin{equation}
\label{eqc66}
h(\tau)>0,\qquad \alpha(t)\sigma^p(t) \le \frac{(q-1)\beta(t)}{[q\zeta(t)]^p},
\qquad t\ge \tau,\qquad q>1,
\end{equation}
where
$$
\sigma(t) = e^{-\int_{t-\tau}^t \gamma(\xi)\, d\xi},\quad 
\zeta(t)=\frac{h(\tau)\nu(\tau)}{\nu(t)} + \frac{\int_\tau^t \beta(\xi)\nu(\xi)\, d\xi}{\nu(t)},\quad t\ge \tau, \quad \nu(t)=e^{-\int_0^t\gamma(\xi)d\xi}.
$$
Then the solution $u(t)$ to equation \eqref{eq1} exists globally. 

If
\begin{equation}
\label{eq50}
M =\sup_{t\ge \tau}\int_\tau^t \gamma(\xi)\, d\xi < \infty,\qquad \int_\tau^\infty \beta(\xi)\nu(\xi)\, d\xi < \infty,
\end{equation}
then the solution $u(t)$ to problem \eqref{eq1} is bounded. 

If 
\begin{equation}
\label{eq51}
\lim_{t\to\infty}\int_\tau^t \gamma(\xi)\, d\xi = -\infty
\end{equation}
and either
\begin{equation}
\label{eqc69}
\qquad \int_\tau^\infty \beta(\xi)\nu(\xi)\, d\xi < \infty \qquad {\rm or}\qquad \lim_{t\to\infty}\frac{\beta(t)}{\gamma(t)} = 0,
\end{equation}
then 
\begin{equation}
\label{eq52}
\lim_{t\to\infty} u(t) = 0.
\end{equation}
\end{theorem}

\begin{proof}
Let $g(t) = \|u(t)\|$. Then from Lemma \ref{compare} we have
\begin{equation}
\|u(t)\| = g(t) \le h(t),\qquad t\ge \tau. 
\end{equation}
This and Lemma \ref{lemma2} imply
\begin{equation}
\label{eq54}
\|u(t)\| \le h(t)\le q \zeta(t) = \frac{h(\tau)\nu(\tau)}{\nu(t)} + \frac{\int_\tau^t \beta(\xi)\nu(\xi)\,d\xi}{\nu(t)},\qquad t\ge \tau. 
\end{equation}
It follows from \eqref{eq54} and Assumption A that the solution $u(t)$ to equation \eqref{eq1} exists globally. 

If relation \eqref{eq50} hold, then
\begin{equation}
\nu(t) = e^{-\int_0^t \gamma(\xi)\, d\xi} \ge e^{-M},\qquad \forall t\ge \tau. 
\end{equation}
This, inequality \eqref{eq54}, and the second inequality in \eqref{eq50} imply
\begin{equation}
\|u(t)\| \le e^M \bigg(h(\tau)\nu(\tau) + \int_0^\infty \beta(\xi)\nu(\xi)\, d\xi \bigg) < \infty.
\end{equation}
This means the solution $u(t)$ is bounded on $\mathbb{R}_+$.

If relation \eqref{eq51} holds, then one gets
\begin{equation}
\label{eq57}
\lim_{t\to\infty} \nu(t) = \lim_{t\to\infty} e^{-\int_\tau^t \gamma(\xi)\, d\xi} = \infty. 
\end{equation}
We claim that 
\begin{equation}
\label{eq58}
\lim_{t\to\infty} \frac{\int_\tau^t \beta(\xi)\nu(\xi)\, d\xi}{\nu(t)} = 0.
\end{equation}
Indeed, if $\int_\tau^\infty \beta(\xi)\nu(\xi)\, d\xi <\infty$, then relation \eqref{eq58} follows from \eqref{eq57}. 
If $\int_\tau^\infty \beta(\xi)\nu(\xi)\, d\xi = \infty$, then relation \eqref{eq58} follows from the second relation in \eqref{eqc69} and L'Hospital's rule. 

From inequality \eqref{eq54} and equalities \eqref{eq57} and \eqref{eq58} one obtains
\begin{equation}
 \lim_{t\to\infty}\|u(t)\| \le  \lim_{t\to\infty} \frac{h(\tau)\nu(\tau)}{\nu(t)}  + 
 \lim_{t\to\infty} \frac{\int_\tau^t \beta(\xi)\nu(\xi)\,d\xi}{\nu(t)} = 0.
 \end{equation} 
 Thus, equality \eqref{eq52} holds. Theorem \ref{thm2} is proved.  
\end{proof}

If $f \equiv 0$, or equivalently $\beta(t)\equiv 0$, then the second inequality in \eqref{eqc66} implies that $\alpha(t)\equiv 0$. Thus, the result in Theorem \ref{thm2} is not very interesting if $f \equiv 0$. However, if the functions $\gamma(t)$ and $\beta(t)$ satisfy either inequalities \eqref{eq50} or inequalities \eqref{eq51} and \eqref{eqc69}, then the global existence and boundedness of the solution $u(t)$ can be obtained if $\alpha(t)$ is sufficiently small so that the second inequality in \eqref{eqc66} holds.

Let us consider equation \eqref{eq1} with $f=0$:
\begin{subequations}
\label{eq14x}
\begin{align}
\dot{u} &= A(t)u + G(t,u(t-\tau)),\qquad t\ge 0,\qquad \dot{u} := \frac{d u}{dt},\\
u(t) &= v(t),\qquad -\tau\le t\le 0,\qquad v(t)\in C([-\tau,0];\mathcal{H}),\qquad \tau=const>0.
\end{align}
\end{subequations}

For equation \eqref{eq14x} we use the following assumption:

{\bf Assumption B:}

\noindent{\it The equation
\begin{subequations}
\begin{align}
\dot{u} &= A(t)u + G(t,u(t-\tau)),\qquad t\ge a,\qquad 
\dot{u} := \frac{d u}{dt},\\
u(t) &= v_{a}(t),\qquad a-\tau\le t\le a,\qquad 
v_a \in C([a-\tau,a];\mathcal{H}),\qquad \tau=const >0,
\end{align}
\end{subequations}
has a unique local solution for all $a\ge 0$ and $v_a \in C([a-\tau,a];\mathcal{H})$.}

Using Theorem \ref{thm1} with $\omega = 0$ for equation \eqref{eq14x} we have the following corollary:
\begin{corollary}
\label{corollary1}
Let Assumption {\bf B} hold and $u(t)$ be the solution to \eqref{eq14x}. Assume that
\begin{gather}
\label{eq62} \frac{1}{\big[\tilde{h}_{\tau}\nu(\tau)\big]^{p-1}} > (p-1) \int_\tau^\infty \frac{\alpha(\xi) \sigma^p(\xi)}{\nu^{p-1}(\xi)}\, d\xi,\qquad \sigma(t) := e^{-\int_{t-\tau}^t \gamma(\xi)\, d\xi},\\
\label{eq63} 
\tilde{h}_{\tau} := \frac{\|v(0)\|}{\nu(\tau)} + \frac{\int_0^\tau \alpha(\xi)\nu(\xi)\|v(\xi-\tau)\|^p \,d\xi}{\nu(\tau)},\qquad \nu(t) = e^{-\int_0^t \gamma(\xi)\,d\xi}.
\end{gather}
Then the solution to problem \eqref{eq14x} exists globally and
\begin{equation}
\label{eqx79}
\|u(t)\| \le \frac{1}{\nu(t)} \bigg(\frac{1}{\big[\tilde{h}_\tau\nu(\tau)\big]^{1-p} - (p-1)\int_\tau^t \frac{\alpha(\xi)\sigma^p(\xi)}{\nu^{p-1}(\xi)}\, d\xi}\bigg)^{\frac{1}{p-1}},\qquad t\ge \tau. 
\end{equation}

In addition, if 
\begin{equation}
\label{eq65}
\sup_{t\ge 0} \int_0^t \gamma(\xi)\, d\xi < \infty,
\end{equation}
then the solution $u(t)$ is bounded. 

If
\begin{equation}
\label{eq66}
\lim_{t\to\infty} \int_0^t \gamma(\xi)\, d\xi = -\infty,
\end{equation}
then
\begin{equation}
\label{eq67}
\lim_{t\to\infty} u(t) = 0.
\end{equation}
\end{corollary}

If $\|v(t)\|$ is sufficiently small on $[-\tau,0]$, then number
$$
\tilde{h}_{\tau} := \frac{\|v(0)\|}{\nu(\tau)} + \frac{\int_0^\tau \alpha(\xi)\nu(\xi)\|v(\xi-\tau)\|^p \,d\xi}{\nu(\tau)}
$$
is sufficiently small and, therefore, the quotient $1/\big[\tilde{h}_{\tau}\nu(\tau)\big]^{p-1}$ is sufficiently large. Thus, inequality \eqref{eq62} holds if  $\|v(t)\|$ is sufficiently small on $[-\tau,0]$. 
If relation \eqref{eq65} holds, then the function $\frac{1}{\nu(t)} =  e^{\int_0^t \gamma(\xi)\, d\xi}$ is bounded. 
Therefore, the right-hand side of \eqref{eqx79} can be made arbitrarily small by making $\tilde{h}_\tau$ sufficiently small. Hence, it follows from \eqref{eqx79} that $\sup_{t\ge 0}\|u(t)\|$ can be made arbitrarily small by making $\|v(t)\|$ sufficiently small on $[-\tau,0]$. If \eqref{eq62} and \eqref{eq66} hold, then  relation \eqref{eq67} holds. These arguments and Corollary \ref{corollary1} yield the following result:

\begin{corollary}
Let Assumption B hold. Assume that
\begin{gather}
\label{eq68}
\int_\tau^\infty \frac{\alpha(\xi)\sigma^p(\xi)}{\nu^{p-1}(\xi)}\, d\xi < \infty,\quad \nu(t) := e^{-\int_0^t \gamma(\xi)\, d\xi},\quad \sigma(t) := e^{-\int_{t-\tau}^t \gamma(\xi)\, d\xi}.
\end{gather} 
If 
\begin{equation}
\label{eq80}
\sup_{t\ge 0}\int_0^t \gamma(\xi)\, d\xi < \infty,
\end{equation}
then the equilibrium solution $u=0$ to equation \eqref{eq14x} is Lyapunov stable. 
If 
\begin{equation}
\lim_{t\to \infty}\int_0^t \gamma(\xi)\, d\xi = - \infty,
\end{equation}
then the equilibrium solution $u=0$ to problem \eqref{eq14x} is asymptotically stable, i.e., $\lim_{t\to\infty} u(t) = 0$ if $\|u(t)\|$ is sufficiently small on $[-\tau,0]$. 
\end{corollary}


Using arguments similar to the ones in \cite{R5}, one obtains the following result:

\begin{lemma}
\label{ramm1}
 If there exists a function $\mu(t)>0$, defined for all $t\ge -\tau$, such that
\begin{subequations}
\label{eqxx20}
\begin{align}
\alpha(t)\frac{\mu(t)}{\mu^p(t-\tau)} +\beta(t)\mu(t)&\le -\gamma(t) - \frac{\dot{\mu}(t)}{\mu(t)},\qquad t\ge 0,\\
\|u(t)\| &\le \frac{1}{\mu(t)},\qquad t\in [-\tau,0],
\end{align}
\end{subequations}
then the solution to equation \eqref{eq1} exists for all $t\ge 0$ and
$$
\|u(t)\| \le \frac{1}{\mu(t)},\qquad t\ge 0.
$$
\end{lemma}

As we have mentioned earlier, although results similar to Lemma \ref{ramm1} are quite general, their applications rely on the existence of $\mu(t)>0$ satisfying inequality \eqref{eqxx20}. Theoretically, Lemma \ref{ramm1} is applicable as long as the solution to equation \eqref{eq8} exists globally. If this is the case, then one can take $\mu(t)= \frac{1}{h(t)}$ and this function satisfies \eqref{eqxx20}. However, it is not known if there is an explicit formula for the solution $h(t)$ to equation \eqref{eq8}. 
Moreover, if $\alpha(t)$, $\beta(t)$, and $\gamma(t)$ are some general functions, it is not easy to find $\mu(t)$ which solves inequality \eqref{eqxx20}.


\end{document}